\documentclass[12pt, reqno]{amsart}
\pdfoutput=1
\usepackage{amsmath, amsthm, amscd, amsfonts, amssymb, graphicx, color, mathrsfs}
\usepackage[style = numeric, sorting = none, maxbibnames = 4]{biblatex}
\usepackage[bookmarksnumbered, colorlinks]{hyperref}
\usepackage[all]{xy}
\usepackage{slashed}
\usepackage{tipa}
\usepackage{soul}
\usepackage{cancel}
\usepackage[normalem]{ulem} 
\usepackage{orcidlink}
\usepackage[noabbrev, capitalize]{cleveref}
\crefname{subsection}{Subsection}{Subsections}

\newcommand{\N}{\mathbb{N}}
\newcommand{\Z}{\mathbb{Z}}

\newcommand{\R}{\mathbb{R}}
\newcommand{\C}{\mathbb{C}}

\DeclareMathOperator*{\esssup}{ess\,sup}

\newcommand{\abs}[1]{\left| #1 \right|}

\newcommand{\meas}[1]{\left| #1 \right|}

\newcommand{\charFunction}[1]{\chi_{#1}}

\newcommand{\innerproduct}[2]{\left\langle #1, #2 \right\rangle}
\newcommand{\innerproductOmega}[2]{\innerproduct{#1}{#2}_{L^{2}(\Omega)}}
\newcommand{\complexConjugate}[1]{\overline{#1}}

\newcommand{\LpnormOver}[3]{\lVert #3 \rVert_{L^{#1}(#2)}}
\newcommand{\lpnormOver}[3]{\lVert #3\rVert_{\ell^{#1}(#2)}}

\newcommand{\operatorNorm}[1]{\lVert#1\rVert_{\mathcal{L}(L^{p}(\Omega),L^{q}(\Omega))}}
\newcommand{\adjOperatorNorm}[1]{\lVert#1\rVert_{\mathcal{L}(L^{q'}(\Omega),L^{p'}(\Omega))}}
\newcommand{\adj}[1]{#1^{*}}
\newcommand{\operatorNormLpLq}[3]{\lVert#3\rVert_{\mathcal{L}(L^{#1}(\Omega),L^{#2}(\Omega))}}
\newcommand{\normIn}[2]{\lVert #1 \rVert_{#2}}

\newcommand{\contLinLpLq}[2]{\mathcal{L}(L^{#1}(\Omega),L^{#2}(\Omega))}

\newcommand{\transp}[1]{#1^{\mathrm{T}}}
\newcommand{\invtransp}[1]{#1^{-\mathrm{T}}}

\newcommand{\intOver}[2]{\int_{#1} #2 \, \mathrm{d}x}
\newcommand{\intOverVar}[3]{\int_{#1} #2 \, \mathrm{d}#3}

\newcommand{\dualLat}[1]{#1^{\perp}}
\newcommand{\dualLatMat}[1]{#1^{\perp}}

\newcommand{\homCo}[1]{\mathbf{#1}}

\newcommand{\dualLatSpace}{\mathbb{H}}
\newcommand{\embRealHyperpl}{\R_{H}^{d+1}}
\newcommand{\embAdLat}{\Z_{H}^{d+1}}

\newcommand{\FourierCoef}[1]{\widehat{#1}}

\newcommand{\FTFund}{\mathcal{F}_{\Omega}}
\newcommand{\FTTor}{\mathcal{F}_{\mathbb{T}^{d}}}

\newcommand{\torus}{\mathbb{T}^{d}}
\newcommand{\torusDimOne}{\mathbb{T}}

\newcommand{\schwartzSpace}{\mathcal{S}}

\newtheorem{theorem}{Theorem}[section]

\theoremstyle{definition}
\newtheorem{definition}[theorem]{Definition}
\newtheorem{example}[theorem]{Example}

\theoremstyle{remark}
\newtheorem{remark}[theorem]{Remark}
\numberwithin{equation}{section}

\DeclareFontFamily{U}{mathx}{\hyphenchar\font45}
\DeclareFontShape{U}{mathx}{m}{n}{
      <5> <6> <7> <8> <9> <10>
      <10.95> <12> <14.4> <17.28> <20.74> <24.88>
      mathx10
      }{}
\DeclareSymbolFont{mathx}{U}{mathx}{m}{n}
\DeclareFontSubstitution{U}{mathx}{m}{n}
\DeclareMathAccent{\widecheck}{0}{mathx}{"71}
\DeclareMathAccent{\wideparen}{0}{mathx}{"75}

\renewbibmacro{in:}{} 
\nocite{*} 

\begin{document}
\setcounter{page}{1}

\title[Boundedness of Fourier multipliers on Fundamental domains]{\(L^{p}\)-\(L^{q}\) boundedness of Fourier multipliers on Fundamental domains of Lattices in $\mathbb{R}^d$}

\author[A. Hendrickx]{Arne Hendrickx \orcidlink{0000-0002-4537-8627}}
\address{
 Arne Hendrickx
  \endgraf
  Department of Mathematics: Analysis, Logic and Discrete Mathematics
  \endgraf
  Ghent University, Belgium
  \endgraf
  {\it E-mail address:} {\rm arnhendr.Hendrickx@UGent.be}
  }



\begin{abstract}
    In this paper we study the \(L^{p}\)-\(L^{q}\) boundedness of Fourier multipliers on the fundamental domain of a lattice in \(\R^{d}\) for \(1 < p,q < \infty\) under the classical Hörmander condition. First, we introduce Fourier analysis on lattices and have a look at possible generalisations. We then prove the Hausdorff-Young inequality, Paley's inequality and the Hausdorff-Young-Paley inequality in the context of lattices. This amounts to a quantitative version of the \(L^{p}\)-\(L^{q}\) boundedness of Fourier multipliers. Moreover, the Paley inequality allows us to prove the Hardy-Littlewood inequality.
\end{abstract}

\maketitle

\tableofcontents

\section{Introduction}

The main goal of this paper is to prove the \(L^{p}\)-\(L^{q}\) boundedness of Fourier multipliers on a fundamental domain \(\Omega\) of a lattice \(L\) in \(\R^{d}\). The Fourier analysis on lattices is very similar to the toroidal case, which corresponds to the theory of the usual Fourier series. In fact, lattices and their fundamental domains can be viewed as linear deformations of the Euclidean lattice \(\Z^{d}\) and the torus \(\torus\) (which can be viewed as a fundamental domain of \(\Z^{d}\)), respectively. Since we start with the basics, this paper may be suitable for anyone new in this field, especially because this paper is essentially self-contained. It may also be interesting to consider this theory from the viewpoint of the more general theory of pseudo-differential operators and Fourier analysis on groups.

This may be the first text on this particular topic. However, pseudo-differential operators on the lattice \(\Z^{d}\) and on the torus \(\torus\) have already been studied in \cite{BKR} and \cite[Chapter 4]{PsiDOAndSymm}, respectively. There are a lot of results available on the \(L^{p}\)-\(L^{q}\) boundedness of Fourier multipliers. A fundamental article on Fourier multipliers on the Euclidean spaces is \cite{Hormander1960} by Lars Hörmander. The results of this article are extended to Fourier multipliers on \(\rm{SU}(2)\) in \cite{ARN1}, on compact homogeneous manifolds in \cite{ARN}, on Riemannian symmetric spaces of the noncompact type in \cite{Anker}, on smooth manifolds in \cite{CKNR}, on compact Lie groups in \cite{Ruzwirth}, on locally compact groups in \cite{AR}, on compact hypergroups in \cite{KR}, to Fourier multipliers associated with the anharmonic oscillator in \cite{CK,CK2} and to Fourier multipliers associated with a generalised $(k,a)$-Fourier transform in \cite{KR1}, just to mention a few of them.

\subsection{Organisation of this paper} In \cref{section:FAFund} we explore Fourier analysis on a lattice \(L\) in \(\R^{d}\). Our first task is to ensure that we have some form of Fourier analysis on a fundamental domain \(\Omega\) of \(L\). This will be the content of \cref{th:FugledeTheorem} of Fuglede, which roughly relates Fourier analysis on a subset \(\Omega \subseteq \R^{d}\) with tilings of \(\R^{d}\) with the lattice \(L\). The notion of the dual lattice \(\dualLat{L}\) shows up here, as this occurs in the explicit description of the orthonormal basis of \(L^{2}(\Omega)\) which enables Fourier analysis on \(\Omega\). Next, we introduce some function spaces, and we construct the Fourier transform \(\FTFund\) of a function \(f \in C^{\infty}(\Omega)\) and look at some properties of it. At the end of this section we invite the reader to extend our theory to the case of spectral sets and to lattices living in proper subspaces of \(\R^{d}\), for which lattices of type \(A_{d}\) can be a guiding example.

In \cref{section:FundIneq} we prove some important inequalities that will help to prove the \(L^{p}\)-\(L^{q}\) boundedness of Fourier multipliers on \(\Omega\). A general result of Hilbert space theory yields Plancherel's identity \eqref{Plancherel} in our context of lattices. We obtain the Hausdorff-Young inequality for \(\FTFund\) by interpolating Plancherel's formula and the estimate \(\lpnormOver{\infty}{\dualLat{L}}{\FourierCoef{f}} \leq \LpnormOver{1}{\Omega}{f}\), where we denote \(\FTFund{f}(\kappa)\) by \(\FourierCoef{f}(\kappa)\).

\Cref{subsection:PaleyType,subsection:HYPIneq} are devoted to a generalisation of some classical results in \cite{Hormander1960} by Hörmander. The first result that we mention is Paley's inequality.
\begin{theorem}[{\cite[Theorem 1.10]{Hormander1960}}]\label{th:HormanderPaley}
    Let \(\varphi > 0\) be a measurable function such that
    \[m\{\xi \in \R^{d} : \varphi(\xi) \geq s\} \leq C/s \quad \text{for all} \quad s > 0,\]
    where \(C > 0\) is some constant and \(m\) is the Lebesgue measure on \(\R^{d}\). Then for any \(1 < p \leq 2\) we have
    \[\left(\intOverVar{\R^{d}}{\abs{\frac{\FourierCoef{f}(\xi)}{\varphi(\xi)}}^{p} \varphi(\xi)^{2}}{\xi}\right)^{\frac{1}{p}} \leq C_{p} \LpnormOver{p}{\R^{d}}{f} \quad \text{for all} \quad f \in L^{p}(\R^{d}),\]
    where \(C_{p} > 0\) only depends on \(p\) and \(C\). Here \(\FourierCoef{f}(\xi)\) denotes the usual Euclidean Fourier transform of \(f\).
\end{theorem}
In our situation we consider the measure space \(\dualLat{L}\) endowed with the counting measure. It is natural to consider this measure, because the discrete version of integration is summation. The translation of \cref{th:HormanderPaley} in our setting then suggests that
\[\left( \sum_{\kappa \in \dualLat{L}} |\FourierCoef{f}(\kappa)|^{p} \, \varphi(\kappa)^{2-p} \right)^{\frac{1}{p}} \lesssim M_{\varphi}^{\frac{2-p}{p}} \LpnormOver{p}{\Omega}{f} \quad \text{for all} \quad f \in L^{p}(\Omega)\]
under the condition that \(\varphi\) is a positive function on \(\dualLat{L}\) satisfying
\[M_{\varphi} := \sup_{s > 0} s \sum\limits_{\substack{\kappa \in \dualLat{L} \\ \varphi(\kappa) \geq s}} 1 < \infty.\]
Here the notation \(\lesssim\) means that the (weak) inequality holds if we multiply the right-hand side with a certain positive constant depending only on the exponents of the relevant Lebesgue spaces, so in this case the constant only depends on \(p\). The factor \(M_{\varphi}^{\frac{2-p}{p}}\) is written explicitly since this embodies the dependence on \(\varphi\). This Paley-type inequality is proven in \cref{th:PaleyType} in \cref{subsection:PaleyType}.

In \cref{subsection:HYPIneq} we discuss the analogue of Hörmander's classical version of the Hausdorff-Young-Paley inequality, which generalises both the Hausdorff-Young inequality and Paley's inequality.
\begin{theorem}[{\cite[Corollary 1.6]{Hormander1960}}]
    Assume \(\varphi\) satisfies the condition of \cref{th:HormanderPaley}. Let \(1 < p \leq 2\) and \(1 < p \leq b \leq p' < \infty\), where \(p'\) denotes the conjugate exponent of \(p\), i.e.\ \(\frac{1}{p} + \frac{1}{p'} = 1\). Then for all \(f \in L^{p}(\R^{d})\) we have
    \[\left(\intOverVar{\R^{d}}{\abs{\FourierCoef{f}(\xi) \, \varphi(\xi)^{\frac{1}{b} - \frac{1}{p'}}}^{b}}{\xi}\right)^{\frac{1}{b}} \leq C_{p} \LpnormOver{p}{\R^{d}}{f}.\]
\end{theorem}
Note that this inequality reduces to the Hausdorff-Young inequality for \(b = p'\) and to Paley's inequality for \(b = p\). We can interpret the Hausdorff-Young-Paley inequality in our context in the same way as for Paley's inequality. The resulting hypothesis is that
\[\left(\sum_{\kappa \in \dualLat{L}} \abs{\FourierCoef{f}(\kappa) \, \varphi(\kappa)^{\frac{1}{b} - \frac{1}{p'}}}^{b}\right)^{\frac{1}{b}} \lesssim M_{\varphi}^{\frac{1}{b} - \frac{1}{p'}} \LpnormOver{p}{\Omega}{f},\]
where \(\varphi\) is a positive function on \(\dualLat{L}\) satisfying
\[M_{\varphi} := \sup_{s > 0} s \sum\limits_{\substack{\kappa \in \dualLat{L} \\ \varphi(\kappa) \geq s}} 1 < \infty.\]
Again, we included the factor \(M_{\varphi}^{\frac{1}{b} - \frac{1}{p'}}\) in the right-hand side of the inequality as it contains the dependence on \(\varphi\). We will show that this assertion is true in \cref{th:HYP}.

In \cref{subsection:HardyLittleIneq} we prove the Hardy-Littlewood inequality as a consequence of the Paley inequality. The original result of Hardy and Littlewood appeared in \cite{HardyLittlewood}.
\begin{theorem}[{\cite[Theorem 5]{HardyLittlewood}}]
    Let \(1 < p \leq 2\). For every \(f \in L^{p}(\torusDimOne)\) we have
    \begin{equation}\label{ineq:HardyLittleOrig}
        \sum_{m = -\infty}^{+\infty} |\FourierCoef{f}(m)|^{p} \, (\abs{m}+1)^{p-2} \lesssim \LpnormOver{p}{\torusDimOne}{f}^{p}.
    \end{equation}
\end{theorem}
We can view \(\torusDimOne\) as the fundamental domain of the one-dimensional lattice \(\Z\), which enables us to interpret this inequality in our setting, where we allow more general weight functions than \((\abs{m}+1)^{p-2}\) in \eqref{ineq:HardyLittleOrig}. Such a weight function \(\varphi\) needs to decay sufficiently fast, which we express by asking that
\[\sum_{\kappa \in \dualLat{L}} \frac{1}{\varphi(\kappa)^{\beta}} < \infty \quad \text{for some} \quad \beta > 0.\]
Under this condition we prove in \cref{th:HardyLittle} that for every \(f \in L^{p}(\Omega)\) we have
\[\left(\sum_{\kappa \in \dualLat{L}} |\FourierCoef{f}(\kappa)|^{p} \, \varphi(\kappa)^{\beta(p-2)}\right)^{\frac{1}{p}} \lesssim_{p,\varphi} \LpnormOver{p}{\Omega}{f}.\]
The hidden constant in this inequality indeed depends on both \(p\) and \(\varphi\), but the dependence on \(\varphi\) can be eliminated by including a similar constant as in the Paley inequality.

Finally, in \cref{section:BoundFourierMult} we prove the \(L^{p}\)-\(L^{q}\) boundedness of Fourier multipliers on \(\Omega\) for \(1 < p,q < \infty\). Hörmander proved in the Euclidean case \cite[Theorem 1.11]{Hormander1960} that a Fourier multiplier with symbol \(\sigma(\xi)\) has a bounded \(L^{p}\)-\(L^{q}\) extension as a consequence of the Hausdorff-Young-Paley inequality, provided that \(1 < b < \infty\), \(1 < p \leq 2 \leq q < \infty\) with \(\frac{1}{p} + \frac{1}{q} = \frac{1}{b}\) and \(\sigma(\xi)\) is a measurable function satisfying
\[m\{\xi \in \R^{d} : \abs{\sigma(\xi)} \geq s\} \leq \frac{C}{s^{b}} \quad \text{for all} \quad s > 0,\]
where \(C > 0\) is some constant. We will prove a similar result in \cref{th:boundednessForPTwo,th:boundedness1}. The fact that \(\Omega\) has finite measure will allow us to obtain also \(L^{p}\)-\(L^{q}\) boundedness results when \(1 < p,q \leq 2\) or \(2 \leq p,q < \infty\) in \cref{th:boundedness2}. The remaining case \(1 < q \leq 2 \leq p < \infty\) will follow from \cref{th:boundedness1} by duality.

\subsection{Notation and conventions}

We follow the convention that \(0 \in \N\).

Throughout this paper \(L\) stands for a lattice in \(\R^{d}\) and \(\Omega\) denotes a fundamental domain of \(L\). These concepts are introduced in \cref{section:FAFund}.

For vectors \(x,y \in \R^{d}\) we write \(x \cdot y = \sum_{j=1}^{d} x_{j} y_{j}\) for the Euclidean inner product and \(\abs{x} = \sqrt{x \cdot x}\) for the Euclidean norm.

Let \(X\) and \(Y\) be normed vector spaces. We denote by \(\mathcal{L}(X,Y)\) the normed vector space of all bounded linear mappings from \(X\) to \(Y\).

If we consider the Lebesgue space \(\ell^{p}(A)\) for some set \(A\) and some \(1 \leq p \leq \infty\), it is always tacitly assumed that the measure on \(A\) is the counting measure. When we sum over the empty set, this is by definition equal to \(0\).

We denote the conjugate exponent of a real number \(1 \leq p \leq \infty\) by \(p'\), i.e.\ \(\frac{1}{p} + \frac{1}{p'} = 1\).

Let \(f \in L^{p}(X,\mu)\) and \(g \in L^{q}(Y,\nu)\). We write \(\normIn{f}{L^{p}(X,\mu)} \lesssim \normIn{g}{L^{q}(Y,\nu)}\) if there exists some constant \(C_{p,q} > 0\) depending only on the exponents of the relevant Lebesgue spaces such that \(\normIn{f}{L^{p}(X,\mu)} \leq C_{p,q} \normIn{g}{L^{q}(Y,\nu)}\). If we want to make this explicit or include other parameters, then we write the parameters as indices, e.g.\ \(\normIn{f}{L^{p}(X,\mu)} \lesssim_{p,\varphi} \normIn{g}{L^{q}(Y,\nu)}\). This notation will also be used in related contexts, such as for norm inequalities for some \(A \in \mathcal{L}(L^{p}(X,\mu),L^{q}(Y,\nu))\).

\section{Fourier analysis on fundamental domains of lattices in \texorpdfstring{$\mathbb{R}^d$}{Rd}} \label{section:FAFund}

In this section we present the basics of Fourier analysis on fundamental domains of lattices in \(\R^{d}\). We refer to \cite{LiXu} and \cite{Fuglede} for more details on this topic.

A lattice \(L\) in \(\R^{d}\) is a discrete subgroup of \(\R^{d}\) spanned by \(d\) linearly independent vectors, which means
\[L = \{k_{1} a_{1} + k_{2} a_{2} + ... + k_{d} a_{d} \mid k_{i} \in \Z\}\]
for some linearly independent column vectors \(a_{1}, \dots, a_{d}\) in \(\R^{d}\). The \(d \times d\)-matrix \(A\) with \(a_{1}, \dots, a_{d}\) as column vectors is called the generator matrix of \(L\) as
\[L = A \Z^{d} = \{Ak \mid k \in \Z^{d}\}.\]
In this case we sometimes denote \(L\) by \(L_{A}\).

A \emph{fundamental set} of a lattice \(L\) in \(\R^{d}\) is a set \(\Omega \subseteq \R^{d}\) such that
\(\Omega + L = \R^{d}\) as a direct sum,
which means that \(\Omega\) contains one representative for every coset of \(\R^{d} / L\). We can also write down this condition by means of the characteristic function \(\charFunction{\Omega}\) of \(\Omega\) as
\[\sum_{\lambda \in L} \charFunction{\Omega}(x+\lambda) = 1 \quad \text{for all } x \in \R^{d}.\]
This expresses that the translated copies of \(\Omega\) do not overlap and do not leave gaps.
A \emph{fundamental domain} of \(L\) is a measurable fundamental set of \(L\). Every lattice \(L = L_{A}\) has a natural fundamental domain, namely the parallelotope
\begin{equation}\label{eq:FundDomParal}
    \Omega_{P} := \left\{\sum_{i=1}^{d} t_{i} a_{i} \mid t = (t_{1}, \dots, t_{d}) \in [0,1)^{d}\right\},
\end{equation}
which is bounded. Remark that a fundamental domain of a lattice \(L\) is clearly not unique.

Now consider any fundamental domain \(\Omega\) of \(L\). We can move every set of the countable decomposition \(\{\Omega \cap (\Omega_{P} + \lambda) \mid \lambda \in L\}\) of \(\Omega\) to \(\Omega_{P}\), and this results in \(\Omega_{P}\) since both \(\Omega\) and \(\Omega_{P}\) are fundamental domains of \(L\). Because of the countable additivity and translation invariance of the Lebesgue measure, we easily see that \(\Omega\) and \(\Omega_{P}\) have the same measure. In particular, since \(\Omega_{P}\) has positive finite measure, this is also the case for every fundamental domain of \(L\).

We are also interested in periodic arrangements of subsets of \(\R^{d}\) determined by the points of a lattice which (almost) fill \(\R^{d}\) under a less restrictive condition than for fundamental domains. This can be expressed by the notion of a tiling. We say that a set \(\Omega \subseteq \R^{d}\) tiles \(\R^{d}\) with the lattice \(L\) if
\[\sum_{\lambda \in L} \charFunction{\Omega}(x + \lambda) = 1 \quad \text{for almost all } x \in \R^{d}.\]
This expresses that the union of all translated copies of the set \(\Omega\) almost fills the whole space \(\R^{d}\) with almost no overlaps. We also write this as \(\Omega + L \approx \R^{d}\).

In order to examine Fourier analysis on a fundamental domain \(\Omega\) of a lattice \(L\) in \(\R^{d}\) we need to construct an appropriate orthonormal basis for \(L^{2}(\Omega)\). We note that we endow \(L^{2}(\Omega)\) with the normalized inner product, namely
\[\innerproductOmega{f}{g} := \frac{1}{\meas{\Omega}} \intOver{\Omega}{f(x) \, \complexConjugate{g(x)}}.\]
Here \(\meas{\Omega}\) denotes the Lebesgue measure of \(\Omega\).

The construction of our orthonormal basis for \(L^{2}(\Omega)\) makes use of the concept of a dual lattice. Given a lattice \(L\) in \(\R^{d}\), the dual lattice \(\dualLat{L}\) is defined by
\[\dualLat{L} := \{\kappa \in \R^{d} \mid \forall \lambda \in L: \kappa \cdot \lambda \in \Z\},\]
where \(\kappa \cdot \lambda\) denotes the usual Euclidean inner product of the vectors \(\kappa, \lambda \in \R^{d}\). The generator matrix of \(\dualLat{L_{A}}\) is \(\dualLatMat{A} := \invtransp{A}\), where \(\transp{A}\) denotes the transpose of the matrix \(A\) and \(\invtransp{A} := \transp{(A^{-1})}\). Note that \(\dualLat{(\dualLat{L})} = L\) for any lattice \(L\).

Surprisingly, it turns out that Fourier analysis on a fundamental domain of a lattice and tilings with this lattice are very closely related as demonstrated by the \emph{Fuglede theorem}.

\begin{theorem}[{Fuglede, \cite[Section 6]{Fuglede}}]\label{th:FugledeTheorem}
Let \(\Omega \subseteq \R^{d}\) be a measurable set with \(0 < \meas{\Omega} < \infty\) and let \(L\) be a lattice in \(\R^{d}\). Then \(\Omega + L \approx \R^{d}\) if and only if \(\{e^{2\pi i \kappa \cdot x} \mid \kappa \in L^{\perp}\}\) is an orthonormal basis for \(L^{2}(\Omega)\).
\end{theorem}

The Fuglede theorem thus provides a natural orthonormal basis to work with. This theorem works for any fundamental domain of a lattice \(L\) as it tiles \(\R^{d}\) with \(L\). Moreover, this is in a certain sense the only type of sets that satisfy the tiling condition. By the lemma in \cite[Section 6]{Fuglede} every set \(\Omega \subseteq \R^{d}\) satisfying \(\Omega + L \approx \R^{d}\) differs from a fundamental domain only by a null set.

As a consequence of the Fuglede theorem, any function \(f \in L^{2}(\Omega)\) can be expanded into a Fourier series, namely
\begin{equation}\label{FourierSeriesExpansion}
    f(x) = \sum_{\kappa \in \dualLat{L}} \FourierCoef{f}(\kappa) \, e^{2 \pi i \kappa \cdot x}, \quad \text{where} \quad \FourierCoef{f}(\kappa) = \frac{1}{\meas{\Omega}} \intOver{\Omega}{f(x) \, e^{-2 \pi i \kappa \cdot x}}.
\end{equation}

In summary, choosing a lattice or, equivalently, a generator matrix and fixing a certain fundamental domain of this lattice, the Fuglede theorem enables Fourier analysis on this fundamental domain through exponential functions related to the dual lattice.

\subsection{The Fourier transform on a fundamental domain}\label{subsection:FT}

We now define the Fourier transform on a fundamental domain \(\Omega\) of a lattice \(L\) in \(\R^{d}\) and discuss some properties of it. Formula \eqref{FourierSeriesExpansion} suggests that we can define the Fourier transform \(\FTFund\) on \(L^{2}(\Omega)\) by
\begin{equation}\label{eq:FTFund}
    \FTFund f(\kappa) := \FourierCoef{f}(\kappa) = \frac{1}{\meas{\Omega}} \intOver{\Omega}{f(x) \, e^{-2 \pi i \kappa \cdot x}} \quad \text{for} \quad \kappa \in \dualLat{L}.
\end{equation}

As in the Euclidean case we would like to find a dense subspace of \(L^{2}(\Omega)\) where the Fourier transform has very useful properties. For the Euclidean Fourier transform this space is the class of Schwartz functions \(\schwartzSpace(\R^{d})\). We will base our findings on the \emph{toroidal Fourier transform}, which is actually a special case of Fourier transforms on fundamental domains of lattices in \(\R^{d}\). See \cite[Chapter 3]{PsiDOAndSymm} for a discussion of toroidal Fourier transforms.

Of course, in the setting of Fourier analysis on fundamental domains of lattices in \(\R^{d}\) we consider different function spaces than in the Euclidean setting. Consider a lattice \(L\) with fundamental domain \(\Omega\). Now note that \(\R^{d} / L\) is homeomorphic to \(\torus = \R^{d} / \Z^{d}\) so that \(\Omega\) is compact for the initial topology. This initial topology is defined as the weakest topology that makes the projection mapping \(x \in \Omega \mapsto x + L \in \R^{d}/L\) continuous, and one can easily check that \(\Omega \cong \R^{d}/L\) with this topology. As a consequence, every complex-valued function on \(\Omega\) has compact support. This means that \(C^{\infty}(\Omega) = C_{c}^{\infty}(\Omega)\), where we note that also the topologies coincide. The compactness of \(\Omega\) also implies that \(C^{\infty}(\Omega) = \schwartzSpace(\Omega)\) (with coinciding topologies). Hence, it is a natural choice to consider the Fourier transform on the function space \(C^{\infty}(\Omega)\), which is constructed as follows.

\begin{definition}
    A function \(f : \R^{d} \to \C\) satisfying \(f(x+\lambda) = f(x)\) for all \(x \in \R^{d}\) and \(\lambda \in L\) is called \(L\)-periodic. The space \(C^{m}(\Omega)\) consists of all \(m\)-times continuously differentiable \(L\)-periodic functions \(f\) considered to be defined on \(\Omega\), and \(C^{\infty}(\Omega) := \bigcap_{m \in \N} C^{m}(\Omega)\). The topology on \(C^{\infty}(\Omega)\) is the topology of uniform convergence on compact sets of the functions and all their derivatives.
\end{definition}

We now define the analogue of the Schwartz space \(\schwartzSpace(\R^{d})\) on a lattice \(L\). Since \(L\) is a discrete space, the requirement of smoothness is removed.

\begin{definition}
    Let \(L\) be a lattice. We say that \(f : L \to \C\) belongs to \(\schwartzSpace(L)\) if for all \(N \in \N\) there exists a constant \(C_{N} > 0\) such that for all \(\lambda \in L\) we have
    \[\abs{f(\lambda)} \leq C_{N} (1 + \abs{\lambda}^{2})^{-N/2}.\]
    We endow \(\schwartzSpace(L)\) with the topology generated by the countable family of seminorms \(p_{j}(f) := \sup_{\lambda \in L} \abs{f(\lambda)} (1 + \abs{\lambda}^{2})^{j/2}\) for \(j \in \N\).
\end{definition}

These function spaces appear in the theory of the toroidal Fourier transform. The torus \(\torus\) is namely the quotient \(\R^{d} / \Z^{d}\) of the space \(\R^{d}\) with the lattice \(\Z^{d}\) and can be viewed as the fundamental domain \([0,1)^{d}\) of \(\Z^{d}\) (with the initial topology). The toroidal Fourier transform has the following useful property, which can be directly compared to the Euclidean case.

\begin{theorem}[{\cite[page 301]{PsiDOAndSymm}}]\label{th:toroidalFT}
    Let
    \[\FTTor : C^{\infty}(\torus) \to \schwartzSpace(\Z^{d}): f \mapsto \FourierCoef{f}(\xi) := \intOver{\torus}{f(x) \, e^{-2 \pi i x \cdot \xi}}\]
    be the toroidal Fourier transform. Then \(\FTTor\) is an isomorphism from \(C^{\infty}(\torus)\) to \(\schwartzSpace(\Z^{d})\), which means that it is a linear homeomorphism.
    Its inverse is given by
    \[\FTTor^{-1} : \schwartzSpace(\Z^{d}) \to C^{\infty}(\torus) : f \mapsto \widecheck{f}(x) = \sum_{\xi \in \Z^{d}} f(\xi) \, e^{2 \pi i x \cdot \xi}.\]
\end{theorem}

We can relate the Fourier transform on a fundamental domain \(\Omega\) of a lattice \(L\) in \(\R^{d}\) to this toroidal case. Since \(A : \torus \to A \torus = \Omega\) is a homeomorphism, we find an (isometric) isomorphism \(\alpha : C^{\infty}(\Omega) \to C^{\infty}(\torus)\) with \(\alpha[f](x) := f(Ax)\) for every \(x \in \torus\). Similarly, we see that \(\beta : \schwartzSpace(\Z^{d}) \to \schwartzSpace(\dualLat{L})\) with \(\beta[f](x) := f((\dualLatMat{A})^{-1} x) = f(\transp{A} x)\) for every \(x \in \dualLat{L}\) is also an isomorphism. Note that the isomorphisms \(\alpha\) and \(\beta\) actually just represent a change of variables.

The Fourier transform on the fundamental domain \(\Omega\) is related to the toroidal Fourier transform via the isomorphisms \(\alpha\) and \(\beta\) as represented in the following diagram:
\[\FTFund : C^{\infty}(\Omega) \xrightarrow{\alpha} C^{\infty}(\torus) \xrightarrow{\FTTor} \schwartzSpace(\Z^{d}) \xrightarrow{\beta} \schwartzSpace(\dualLat{L}).\]
Since \(\FTFund = \beta \circ \FTTor \circ \alpha\) is a composition of isomorphisms, \cref{th:toroidalFT} gives the following result.

\begin{theorem}
    Let
    \[\FTFund : C^{\infty}(\Omega) \to \schwartzSpace(\dualLat{L}) : f \mapsto \FourierCoef{f}(\kappa) := \frac{1}{\meas{\Omega}} \intOver{\Omega}{f(x) \, e^{-2 \pi i \kappa \cdot x}}\]
    be the Fourier transform on a fundamental domain \(\Omega\) of a lattice \(L\) in \(\R^{d}\). Then \(\FTFund\) is an isomorphism from \(C^{\infty}(\Omega)\) to \(\schwartzSpace(\dualLat{L})\). Its inverse is given by
    \[\FTFund^{-1} : \schwartzSpace(\dualLat{L}) \to C^{\infty}(\Omega) : f \mapsto \widecheck{f}(x) = \sum_{\kappa \in \dualLat{L}} f(\kappa) \, e^{2 \pi i \kappa \cdot x}.\]
\end{theorem}

Later we will see that both \(\FTFund\) and \(\FTFund^{-1}\) have bounded \(L^{p}(\Omega)\)-\(\ell^{p'}(\dualLat{L})\) extensions for every \(1 \leq p \leq 2\) as a consequence of the Hausdorff-Young inequalities in \cref{th:HY}, where \(\frac{1}{p} + \frac{1}{p'} = 1\). But note that these extensions need not be each other's inverse on a superspace of \(C^{\infty}(\Omega)\). However, we will still write \(\FTFund^{-1}\) for the bounded extension the corresponding operator on \(C^{\infty}(\Omega)\) as this is customary notation.

Note that \(C^{\infty}(\Omega)\) is dense in \(L^{p}(\Omega)\) for \(1 \leq p < \infty\) since this is generally known for \(\Omega \subseteq \R^{d}\). It is also true that \(\schwartzSpace(\dualLat{L})\) is dense in \(\ell^{p}(\dualLat{L})\) for \(1 \leq p < \infty\), because we can approximate an arbitrary \(f \in \ell^{p}(\dualLat{L})\) by a sequence of finitely supported functions on \(\dualLat{L}\) and these belong to \(\schwartzSpace(\dualLat{L})\). Hence, the bounded \(L^{p}(\Omega)\)-\(\ell^{p'}(\dualLat{L})\) extensions of \(\FTFund\) and \(\FTFund^{-1}\) are unique.

\begin{remark}
    Suppose that there exists some null set \(N \subseteq \Omega\) such that \(\widetilde{\Omega} := \Omega \setminus N\) is open. Let \(\tau_{init}\) be the initial topology on \(\Omega\), and let \(\widetilde{\tau}_{rel}\) be the relative topology on \(\widetilde{\Omega}\). The function spaces in this subsection and the results later in this paper all refer to the initial topology on \(\Omega\). The reason for this is that the relative topology \(\tau_{rel}\) on \(\Omega\) may be too strong to allow compactness of \(\Omega\), while the (possibly) weaker topology \(\tau_{init} \subseteq \tau_{rel}\) solves this problem. However, the relative topology is also a natural topology to consider, but are the results in this paper still true for \((\widetilde{\Omega},\tau_{rel})\)?
    
    This happens to be the case if we make the necessary adjustments. Since \((\widetilde{\Omega},\widetilde{\tau}_{rel})\) is not compact, we consider the function space \(C_{c}^{\infty}(\widetilde{\Omega},\widetilde{\tau}_{rel})\) of smooth compactly supported functions on \(\widetilde{\Omega}\). Now note that \(C_{c}^{\infty}(\widetilde{\Omega},\widetilde{\tau}_{rel}) \subseteq C^{\infty}(\Omega,\tau_{init})\). Hence, the Fourier transform \(\FTFund\) is defined on \(C_{c}^{\infty}(\widetilde{\Omega},\widetilde{\tau}_{rel})\), but the restriction of the isomorphism \(\alpha\) is not surjective anymore, so that the image of \(C_{c}^{\infty}(\widetilde{\Omega},\widetilde{\tau}_{rel})\) under \(\FTFund\) is difficult to describe. However, since we do not use the surjectivity of \(\FTFund\) in this paper, all results will remain valid. Remark that \(L^{p}(\widetilde{\Omega}) = L^{p}(\Omega)\) for all \(1 \leq p \leq \infty\). Hence, in particular, we also find that a Fourier multiplier \(A : C_{c}^{\infty}(\widetilde{\Omega},\widetilde{\tau}_{rel}) \to C_{c}^{\infty}(\widetilde{\Omega},\widetilde{\tau}_{rel})\) has a bounded \(L^{p}(\widetilde{\Omega})\)-\(L^{q}(\widetilde{\Omega})\) extension under some conditions from \cref{section:BoundFourierMult}. For example, this reasoning applies when \(\widetilde{\Omega}\) is the interior of \(\Omega_{P}\), where \(\Omega_{P}\) is the parallelotopic fundamental domain as given by equation \eqref{eq:FundDomParal}.
\end{remark}

\subsection{Spectral sets and the Fuglede conjecture}

We now briefly discuss the Fuglede theorem outside the context of lattices. A measurable set \(\Omega \subseteq \R^{d}\) with \(0 < \meas{\Omega} < \infty\) is called a \emph{spectral set} if there exists a set \(\Lambda \subseteq \R^{d}\) such that \(\{e^{2 \pi i \lambda \cdot x} \mid \lambda \in \Lambda\}\) is an orthonormal basis for \(L^{2}(\Omega)\). In this case the set \(\Lambda\) is called an \emph{exponent set} for \(\Omega\). Note that every exponent set is countably infinite.

Further, a measurable set \(\Omega \subseteq \R^{d}\) with \(0 < \meas{\Omega} < \infty\) is called a \emph{direct summand} if there exists a set \(\Gamma \subseteq \R^{d}\) such that (after modifying \(\Omega\) with a null set) \(\Omega + \Gamma = \R^{d}\) as a direct sum. This set \(\Gamma\) is called a \emph{translation set} for \(\Omega\).

It is asserted in \cite[page 119]{Fuglede} that exponent sets and translation sets are always discrete, closed and total, which means that these sets are not contained in a proper subspace of \(\R^{d}\). As a consequence every translation set is also countably infinite. Remark that the Fuglede theorem tells us that a measurable set \(\Omega \subseteq \R^{d}\) is a spectral set admitting an exponent \emph{subgroup} \(\Lambda \subseteq \R^{d}\) if and only if \(\Omega\) is a direct summand admitting a translation \emph{subgroup} \(\Gamma \subseteq \R^{d}\).

In \cite{Fuglede} Fuglede conjectures that, more generally, a measurable set \(\Omega \subseteq \R^{d}\) with \(0 < \meas{\Omega} < \infty\) is a spectral set if and only if \(\Omega\) is a direct summand. \Cref{th:FugledeTheorem} states that this conjecture is true for fundamental domains of lattices in \(\R^{d}\). This conjecture has also been shown in \cite{FugConj} to be true for any compact convex set \(\Omega \subseteq \R^{d}\) with non-empty interior, but counterexamples have been found in \(\R^{d}\) for \(d \geq 3\). However, the conjecture remains unknown for \(d \in \{1,2\}\).

We can introduce a Fourier transform on a spectral set \(\Omega\) in the same way as in \cref{subsection:FT}. We also have an obvious candidate for the inverse Fourier transform. However, we do not have some isomorphisms that link this Fourier transform to a canonical Fourier transform as before. So the question arises whether we can prove that this Fourier transform is invertible and whether it is an isomorphism on \(C_{c}^{\infty}(\Omega)\). This would provide a way to generalise the results in this paper to any spectral set \(\Omega \subseteq \R^{d}\). However, it may seem improbable that we can invert this Fourier transform, because we might loose control over the corresponding Dirichlet kernel if we do not have more structural information about the exponent set \(\Lambda\). Otherwise we could ask whether there is another approach that works for all spectral sets. We leave this question open.

\subsection{Lattices of type \texorpdfstring{\(A_{d}\)}{Ad}}

Sometimes a lattice \(L\) has a simpler description in a higher-dimensional space. One can try to extend the theory in this paper to lattices living in proper subspaces. We will give an example of such a lattice, which appears in \cite{LiXu}. We study this particular kind of lattices through homogeneous coordinates.

Let \(d \geq 1\). If we identify \(\R^{d}\) with the hyperplane
\[\embRealHyperpl := \{(t_{1}, t_{2}, \dots, t_{d+1}) \in \R^{d+1} \mid t_{1} + t_{2} + \dots + t_{d+1} = 0\}\]
in \(\R^{d+1}\), then the lattice of type \(A_{d}\) is defined as
\[\embAdLat := \Z^{d+1} \cap \embRealHyperpl = \{(k_{1}, k_{2}, \dots, k_{d+1}) \in \Z^{d+1} \mid k_{1} + k_{2} + \dots + k_{d+1} = 0\}.\]
Note that this lattice lives in a \(d\)-dimensional subspace of \(\R^{d+1}\), so it is described by homogeneous coordinates \(\homCo{t}\), for which we will use bold letters, and \(A_{d} := \embAdLat\).

We consider the fundamental domain of the lattice \(A_{d}\) tiling \(\embRealHyperpl\) given by
\[\Omega_{H} := \{\homCo{t} \in \R_{H}^{d+1} \mid -1 < t_{i} - t_{j} \leq 1, \forall \, 1 \leq i < j \leq d+1\}.\]
Note that the strict inequality ensures that the translated copies of \(\Omega_{H}\) do not overlap. As an example, for \(d = 2\) this fundamental domain \(\Omega_{H}\) is a regular hexagon.

Note that we need slightly different but similar definitions for the generator matrix of the lattice \(A_{d}\) and related concepts because this lattice lives in a proper subspace of \(\R^{d+1}\). The natural choice for the generator matrix of the lattice \(A_{d}\) is the \((d+1) \times d\) matrix
\[A := \begin{pmatrix}
1 & 0 & \cdots & 0 & 0 \\
0 & 1 & \cdots & 0 & 0 \\
\vdots & \vdots & \ddots & \vdots & \vdots \\
0 & 0 & \cdots & 1 & 0 \\
0 & 0 & \cdots & 0 & 1 \\
-1 & -1 & \cdots & -1 & -1
\end{pmatrix},\]
since \(A_{d} = A \Z^{d+1}\) with this choice.
The dual lattice \(\dualLat{A_{d}}\) is defined to be generated by the matrix
\[\dualLatMat{A} := A(\transp{A} A)^{-1} = \frac{1}{d+1} \begin{pmatrix}
d & -1 & \cdots & -1 & -1 \\
-1 & d & \cdots & -1 & -1 \\
\vdots & \vdots & \ddots & \vdots & \vdots \\
-1 & -1 & \cdots & d & -1 \\
-1 & -1 & \cdots & -1 & d \\
-1 & -1 & \cdots & -1 & -1
\end{pmatrix}.\]

We now give an explicit description of the dual lattice \(\dualLat{A_{d}}\). Consider
\[\dualLatSpace := \{\homCo{k} \in A_{d} \mid k_{1} \equiv k_{2} \equiv \dots \equiv k_{d+1} \pmod{d+1}\}.\]
To every element \(\dualLatMat{A} j\) with \(j \in \Z^{d}\) of the dual lattice \(\dualLat{A_{d}}\) we relate the vector \(\homCo{k} = (d+1) \dualLatMat{A} j\), which belongs to \(A_{d}\) because the integers are closed under linear combinations. Moreover, we can easily compute that \(\homCo{k} \in \dualLatSpace\). Conversely, we immediately see that \(j = \transp{A} \homCo{k} / (d+1)\), so it follows that \(j \in \Z^{d}\) for all \(\homCo{k} \in \dualLatSpace\). In short, this means that
\[\dualLat{A_{d}} = \left\{\dualLatMat{A} j \mid j \in \Z^{d}\right\} = \left\{\frac{\homCo{k}}{d+1} \mid \homCo{k} \in \dualLatSpace\right\}.\]

The Fuglede theorem tells us that \(\{\phi_{\homCo{j}} \mid \homCo{j} \in \dualLatSpace\}\) is an orthonormal basis of \(L^{2}(\Omega_{H})\), where
\[\phi_{\homCo{j}}(\homCo{t}) := e^{\frac{2 \pi i}{d+1} \homCo{j} \cdot \homCo{t}} \quad \text{for} \quad \homCo{j} \in \dualLatSpace \quad \text{and} \quad \homCo{t} \in \embRealHyperpl.\]
In particular, this means that
\[\innerproduct{\phi_{\homCo{j}}}{\phi_{\homCo{k}}} := \frac{1}{\meas{\Omega_{H}}} \intOverVar{\Omega_{H}}{\phi_{\homCo{j}}(\homCo{t}) \, \complexConjugate{\phi_{\homCo{k}}(\homCo{t})}}{\homCo{t}} = \delta_{\homCo{j} \homCo{k}},\]
where \(\delta_{\homCo{j} \homCo{k}} = 1\) if \(\homCo{j} = \homCo{k}\) and \(0\) otherwise. We note that \(\meas{\Omega_{H}} = \sqrt{\det(\transp{A} A)} = \sqrt{d+1}\).

Let us call a function \(f\) on \(\embRealHyperpl\) \(H\)-periodic if it is periodic with respect to the lattice \(A_{d}\), which means that \(f(\homCo{t} + \homCo{k}) = f(\homCo{t})\) for all \(\homCo{k} \in A_{d}\). Note that the functions \(\phi_{\homCo{j}}\) are \(H\)-periodic. Further, we find that every \(H\)-periodic \(L^{2}(\Omega_{H})\)-function \(f\) has a Fourier series expansion:
\[f(\homCo{t}) = \sum_{\homCo{k} \in \dualLatSpace} \FourierCoef{f_{\homCo{k}}} \phi_{\homCo{k}}(\homCo{t}), \quad \text{where } \FourierCoef{f_{\homCo{k}}} := \frac{1}{\sqrt{d+1}} \intOverVar{\Omega_{H}}{f(\homCo{t}) \, \phi_{-\homCo{k}}(\homCo{t})}{\homCo{t}}.\]

\section{Some fundamental inequalities}\label{section:FundIneq}

In this section we prove some classical inequalities in the setting of fundamental domains of lattices in \(\R^{d}\), namely the Hausdorff-Young inequality, Paley's inequality and the Hausdorff-Young-Paley inequality. We will also treat the Hardy-Littlewood inequality, but this is not necessary for the further developments. We begin by deriving the Plancherel formula in our setting.

Let \(L\) be a lattice in \(\R^{d}\) with fundamental domain \(\Omega\). By a well-known general fact about Hilbert spaces \cite[Theorem 4.18]{RealComplex} applied to \(L^{2}(\Omega)\) with orthonormal basis \(\{e^{2 \pi i \kappa \cdot x} \mid \kappa \in \dualLat{L}\}\) we find for \(f \in L^{2}(\Omega)\) the following \emph{Plancherel formula}:
\begin{equation}\label{Plancherel}
    \LpnormOver{2}{\Omega}{f}^{2} = \frac{1}{\meas{\Omega}} \intOver{\Omega}{\abs{f(x)}^{2}} = \sum_{\kappa \in \dualLat{L}} |\FourierCoef{f}(\kappa)|^{2} = \lpnormOver{2}{\dualLat{L}}{\FourierCoef{f}}^{2},
\end{equation}
where \(\FourierCoef{f}(\kappa)\) is the Fourier transform of \(f\) as in equation \eqref{eq:FTFund}.

\subsection{Hausdorff-Young inequality}\label{subsection:HYIneq}

The first inequality that we will consider in this section, is the Hausdorff-Young inequality. It is fundamental since it shows that the Fourier transform and its inverse have bounded \(L^{p}\)-\(L^{p'}\) extensions. The proof proceeds as in the Euclidean case \cite[Corollary 1.3.14]{PsiDOAndSymm}.

\begin{theorem}[Hausdorff-Young inequality]\label{th:HY}
    Let \(1 \leq p \leq 2\) and \(\frac{1}{p} + \frac{1}{p'} = 1\). If \(f \in L^{p}(\Omega)\), then \(\FourierCoef{f} \in \ell^{p'}(\dualLat{L})\) and
    \begin{equation}\label{eq:HY}
        \lpnormOver{p'}{\dualLat{L}}{\FourierCoef{f}} \leq \LpnormOver{p}{\Omega}{f}.
    \end{equation}
    Similarly, if \(f \in \ell^{p}(\dualLat{L})\), then \(\widecheck{f} \in L^{p'}(\Omega)\) and
    \begin{equation}\label{eq:HYInv}
        \LpnormOver{p'}{\Omega}{\widecheck{f}} \leq \lpnormOver{p}{\dualLat{L}}{f}.
    \end{equation}
\end{theorem}

\begin{proof}
    For all \(f \in L^{1}(\Omega)\) we have
    \[\lpnormOver{\infty}{\dualLat{L}}{\FourierCoef{f}} = \sup_{\kappa \in \dualLat{L}} \frac{1}{\meas{\Omega}} \abs{\intOver{\Omega}{f(x) \, e^{-2 \pi i \kappa \cdot x}}} \leq \LpnormOver{1}{\Omega}{f}.\]
    We obtain the desired result by a straightforward application of the Riesz-Thorin interpolation theorem \cite[Theorem 1.3.4]{Grafakos} using this estimate and the Plancherel formula \eqref{Plancherel}.
    
    The second inequality follows similarly by applying Riesz-Thorin interpolation to the estimates
    \[\LpnormOver{\infty}{\Omega}{\widecheck{f}} = \esssup_{x \in \Omega} \abs{\sum_{\kappa \in \dualLat{L}} f(\kappa) \, e^{2 \pi i \kappa \cdot x}} \leq \sum_{\kappa \in \dualLat{L}} \abs{f(\kappa)} = \lpnormOver{1}{\dualLat{L}}{f}\]
    and
    \[\LpnormOver{2}{\Omega}{\widecheck{f}}^{2} = \frac{1}{\meas{\Omega}} \intOver{\Omega}{\abs{\sum_{\kappa \in \dualLat{L}} f(\kappa) \, e^{2 \pi i \kappa \cdot x}}^{2}} \leq \frac{1}{\meas{\Omega}} \intOver{\Omega}{\sum_{\kappa \in \dualLat{L}} \abs{f(\kappa)}^{2}} = \lpnormOver{2}{\dualLat{L}}{f}^{2}.\]
    This completes the proof.
\end{proof}

\subsection{Paley's inequality}\label{subsection:PaleyType}

We now consider a Paley-type inequality, which can be seen as a weighted version of the Plancherel formula \eqref{Plancherel} for \(L^{p}(\Omega)\) with \(1 < p \leq 2\). Our proof strategy is based on the proof of \cite[Theorem 4.2]{CKNR}.

\begin{theorem}[Paley's inequality]\label{th:PaleyType}
    Let \(1 < p \leq 2\). If \(\varphi(\kappa)\) is a positive function on \(\dualLat{L}\) such that
    \[M_{\varphi} := \sup_{s > 0} s \sum\limits_{\substack{\kappa \in \dualLat{L} \\ \varphi(\kappa) \geq s}} 1 < \infty,\]
    then for every \(f \in L^{p}(\Omega)\) we have
    \begin{equation}\label{ineq:PaleyType}
        \left( \sum_{\kappa \in \dualLat{L}} |\FourierCoef{f}(\kappa)|^{p} \, \varphi(\kappa)^{2-p} \right)^{\frac{1}{p}} \lesssim M_{\varphi}^{\frac{2-p}{p}} \LpnormOver{p}{\Omega}{f}.
    \end{equation}
\end{theorem}

\begin{proof}
    Note that the case \(p=2\) is fulfilled by Plancherel's formula \eqref{Plancherel}. So now we suppose \(1 < p < 2\). We define a measure \(\mu\) on \(\dualLat{L}\) by \(\mu\{\kappa\} = \varphi(\kappa)^{2}\) for every \(\kappa \in \dualLat{L}\). Consider the (sub)linear operator \(A\) defined by
    \[Af(\kappa) = \frac{\FourierCoef{f}(\kappa)}{\varphi(\kappa)}.\]
    
    Now, our goal is to show that \(A\) is of weak type \((1,1)\) and \((2,2)\) so that the result follows from Marcinkiewicz' interpolation theorem \cite[Theorem 1.3.2]{Grafakos}. More specifically, we will see for every \(s > 0\) that
    \[\mu\{\kappa \in \dualLat{L} : \abs{Af(\kappa)} \geq s\} \leq \left(\frac{\LpnormOver{2}{\Omega}{f}}{s}\right)^{2}\]
    and
    \[\mu\{\kappa \in \dualLat{L} : \abs{Af(\kappa)} \geq s\} \leq 2 M_{\varphi} \frac{\LpnormOver{1}{\Omega}{f}}{s}.\]
    
    Because of the Plancherel formula \(\eqref{Plancherel}\) we obtain for every \(f \in L^{2}(\Omega)\) that
    \begin{align*}
        s^{2} \mu\{\kappa \in \dualLat{L} : \abs{Af(\kappa)} \geq s\} &= s^{2} \sum_{\substack{\kappa \in \dualLat{L} \\ \abs{Af(\kappa)} \geq s}} \varphi(\kappa)^{2} \\
        &\leq \sum_{\kappa \in \dualLat{L}} \abs{Af(\kappa)}^{2} \varphi(\kappa)^{2} \\
        &= \sum_{\kappa \in \dualLat{L}} |\FourierCoef{f}(\kappa)|^{2} \\
        &= \LpnormOver{2}{\Omega}{f}^{2}.
    \end{align*}
    This proves that \(A\) is of weak type \((2,2)\).
    
    Now note that \(\abs{Af(\kappa)} = |\FourierCoef{f}(\kappa)|/\varphi(\kappa) \leq \LpnormOver{1}{\Omega}{f}/\varphi(\kappa)\) by the Hausdorff-Young inequality \eqref{eq:HY}. Hence for \(f \in L^{1}(\Omega)\) we have
    \[\mu\{\kappa \in \dualLat{L} : \abs{Af(\kappa)} \geq s\} \leq \mu\Bigg\{\kappa \in \dualLat{L} : \varphi(\kappa) \leq \frac{\LpnormOver{1}{\Omega}{f}}{s}\Bigg\} = \sum_{\substack{\kappa \in \dualLat{L} \\ \varphi(\kappa) \leq \sigma}} \varphi(\kappa)^{2},\]
    where we set \(\sigma := \LpnormOver{1}{\Omega}{f}/s\). Next we find
    \[\sum_{\substack{\kappa \in \dualLat{L} \\ \varphi(\kappa) \leq \sigma}} \varphi(\kappa)^{2} = \sum_{\substack{\kappa \in \dualLat{L} \\ \varphi(\kappa) \leq \sigma}} \int_{0}^{\varphi(\kappa)^{2}} \, \mathrm{d}\tau \leq \int_{0}^{\sigma^{2}} \sum_{\substack{\kappa \in \dualLat{L} \\ \tau^{1/2} \leq \varphi(\kappa) \leq \sigma}} 1 \, \mathrm{d}\tau,\]
    where we used Fubini's theorem and enlarged the `domain of double integration', viewing summation as an integral with respect to the counting measure. Finally, the substitution \(t = \sqrt{\tau}\) gives
    \[\int_{0}^{\sigma^{2}} \sum_{\substack{\kappa \in \dualLat{L} \\ \tau^{1/2} \leq \varphi(\kappa) \leq \sigma}} 1 \, \mathrm{d}\tau = 2 \int_{0}^{\sigma} t \sum_{\substack{\kappa \in \dualLat{L} \\ t \leq \varphi(\kappa) \leq \sigma}} 1 \, \mathrm{d}t \leq 2 \int_{0}^{\sigma} t \sum_{\substack{\kappa \in \dualLat{L} \\ t \leq \varphi(\kappa)}} 1 \, \mathrm{d}t \leq 2 M_{\varphi} \sigma.\]
    We thus found for every \(f \in L^{2}(\Omega)\) that
    \[\mu\{\kappa \in \dualLat{L} : \abs{Af(\kappa)} \geq s\} \leq 2 M_{\varphi} \frac{\LpnormOver{1}{\Omega}{f}}{s},\]
    which means that \(A\) is of weak type \((1,1)\).
    
    Remark that the Marcinkiewicz' interpolation theorem \cite[Theorem 1.3.2]{Grafakos} also ensures that the hidden constant in inequality \eqref{ineq:PaleyType} only depends on \(p\), since we eliminated the dependence on \(\varphi\) by explicitly writing the factor \(M_{\varphi}^{\frac{2-p}{p}}\).
\end{proof}

\subsection{Hardy-Littlewood inequality}\label{subsection:HardyLittleIneq}

As an intermezzo we include the Hardy-Littlewood inequality since it follows by a simple application of the Paley inequality. It expresses that we have a weighted version of Plancherel's formula \eqref{Plancherel} for \(L^{p}(\Omega)\) with \(1 < p \leq 2\) if the weight decays sufficiently fast. To this end we apply a similar reasoning as in the proof of \cite[Theorem 3.5]{KR}.

\begin{theorem}[Hardy-Littlewood inequality]\label{th:HardyLittle}
    Let \(1 < p \leq 2\), and let \(\varphi(\kappa)\) be a positive function on \(\dualLat{L}\) growing sufficiently fast in the sense that
    \[\sum_{\kappa \in \dualLat{L}} \frac{1}{\varphi(\kappa)^{\beta}} < \infty \quad \text{for some} \quad \beta > 0.\]
    Then for every \(f \in L^{p}(\Omega)\) it holds that
    \[\left(\sum_{\kappa \in \dualLat{L}} |\FourierCoef{f}(\kappa)|^{p} \, \varphi(\kappa)^{\beta(p-2)}\right)^{\frac{1}{p}} \lesssim_{p,\varphi} \LpnormOver{p}{\Omega}{f}.\]
\end{theorem}

\begin{proof}
    Put
    \[C := \sum_{\kappa \in \dualLat{L}} \frac{1}{\varphi(\kappa)^{\beta}} < \infty.\]
    Then we find for any \(s > 0\) that
    \[C \geq \sum_{\substack{\kappa \in \dualLat{L} \\ \varphi(\kappa)^{\beta} \leq \frac{1}{s}}} \frac{1}{\varphi(\kappa)^{\beta}} \geq s \sum_{\substack{\kappa \in \dualLat{L} \\ \varphi(\kappa)^{\beta} \leq \frac{1}{s}}} 1 = s \sum_{\substack{\kappa \in \dualLat{L} \\ \frac{1}{\varphi(\kappa)^{\beta}} \geq s}} 1\]
    so that
    \[\sup_{s > 0} s \sum_{\substack{\kappa \in \dualLat{L} \\ \frac{1}{\varphi(\kappa)^{\beta}} \geq s}} 1 \leq C < \infty.\]
    Hence, by Paley's inequality \eqref{ineq:PaleyType}, applied to the function \(1/\varphi(\kappa)^{\beta}\), we obtain for every \(f \in L^{p}(\Omega)\) that
    \[\left(\sum_{\kappa \in \dualLat{L}} |\FourierCoef{f}(\kappa)|^{p} \, \varphi(\kappa)^{\beta(p-2)}\right)^{\frac{1}{p}} \lesssim_{p,\varphi} \LpnormOver{p}{\Omega}{f}.\]
    It is clear that the hidden constant in the Hardy-Littlewood inequality indeed also depends on \(\varphi\) since we did not include the factor \(M_{1/\varphi(\kappa)^{\beta}}^{\frac{2-p}{p}}\) from Paley's inequality \eqref{ineq:PaleyType}.
\end{proof}

\subsection{Hausdorff-Young-Paley inequality}\label{subsection:HYPIneq}

We conclude this section with the Hausdorff-Young-Paley inequality. In \cite{Hormander1960} Lars Hörmander applied the Euclidean version of this inequality to prove the \(L^{p}\)-\(L^{q}\) boundedness of Fourier multipliers. It will also play a crucial role in our arguments in the next section.

Before stating the Hausdorff-Young-Paley inequality we mention the following important interpolation result, which we will use in our proof.

\begin{theorem}[{Stein-Weiss, \cite[Corollary 5.5.4]{InterpolSpaces}}]\label{th:realInterpolation}
    Consider two measure spaces \((X,\mu)\) and \((Y,\nu)\). Let \(1 \leq p_{0}, p_{1}, q_{0}, q_{1} < \infty\). Assume that the linear operator \(A\) defined on \(L^{p_{0}}(X, w_{0} \, \mathrm{d}\mu) + L^{p_{1}}(X, w_{1} \, \mathrm{d}\mu)\) satisfies
    \[\LpnormOver{q_{0}}{Y, \widetilde{w}_{0} \, \mathrm{d}\nu}{Af} \leq M_{0} \LpnormOver{p_{0}}{X, w_{0} \, \mathrm{d}\mu}{f} \quad \text{for all} \quad f \in L^{p_{0}}(X, w_{0} \, \mathrm{d}\mu)\]
    and
    \[\LpnormOver{q_{1}}{Y, \widetilde{w}_{1} \, \mathrm{d}\nu}{Af} \leq M_{1} \LpnormOver{p_{1}}{X, w_{1} \, \mathrm{d}\mu}{f} \quad \text{for all} \quad f \in L^{p_{1}}(X, w_{1} \, \mathrm{d}\mu)\]
    for some \(M_{0} > 0\) and \(M_{1} > 0\), where \(w_{j}\) and \(\widetilde{w}_{j}\) are weight functions for \(j \in \{0,1\}\), i.e.\ positive measurable functions. Let \(0 \leq \theta \leq 1\) be arbitrary, and let
    \[\frac{1}{p} = \frac{1-\theta}{p_{0}} + \frac{\theta}{p_{1}} \quad \text{and} \quad \frac{1}{q} = \frac{1-\theta}{q_{0}} + \frac{\theta}{q_{1}}.\]
    Then \(A\) extends to a bounded linear operator \(A : L^{p}(X, w \, \mathrm{d}\mu) \to L^{q}(Y, \widetilde{w} \, \mathrm{d}\nu)\) satisfying
    \[\LpnormOver{q}{Y, \widetilde{w} \, \mathrm{d}\nu}{Af} \leq M_{0}^{1-\theta} M_{1}^{\theta} \LpnormOver{p}{X, w \, \mathrm{d}\mu}{f} \quad \text{for all} \quad f \in L^{p}(X, w \, \mathrm{d}\mu),\]
    where
    \[w = w_{0}^{\frac{p(1-\theta)}{p_{0}}} w_{1}^{\frac{p\theta}{p_{1}}} \quad \text{and} \quad \widetilde{w} = \widetilde{w}_{0}^{\frac{q(1-\theta)}{q_{0}}} \widetilde{w}_{1}^{\frac{q\theta}{q_{1}}}.\]
\end{theorem}

We are now ready to state and prove the Hausdorff-Young-Paley inequality, which is in a sense an interpolated version of the Hausdorff-Young inequality \eqref{eq:HY} and Paley's inequality \eqref{ineq:PaleyType}. We follow the same line of reasoning as in the proof of \cite[Theorem 4.6]{CKNR}.

\begin{theorem}[Hausdorff-Young-Paley inequality]\label{th:HYP}
    Let \(1 < p \leq 2\), and let \(1 < p \leq b \leq p' < \infty\) with \(\frac{1}{p} + \frac{1}{p'} = 1\). If \(\varphi(\kappa)\) is a positive function on \(\dualLat{L}\) such that
    \[M_{\varphi} := \sup_{s > 0} s \sum\limits_{\substack{\kappa \in \dualLat{L} \\ \varphi(\kappa) \geq s}} 1 < \infty,\]
    then for every \(f \in L^{p}(\Omega)\) we have
    \begin{equation}\label{ineq:HYP}
        \left(\sum_{\kappa \in \dualLat{L}} \abs{\FourierCoef{f}(\kappa) \, \varphi(\kappa)^{\frac{1}{b} - \frac{1}{p'}}}^{b}\right)^{\frac{1}{b}} \lesssim M_{\varphi}^{\frac{1}{b} - \frac{1}{p'}} \LpnormOver{p}{\Omega}{f}.
    \end{equation}
\end{theorem}

\begin{proof}
    Plancherel's identity \eqref{Plancherel} treats the case \(p=2\), so assume that \(p < 2\). Paley's inequality \eqref{ineq:PaleyType} ensures that the linear operator \(\FTFund : L^{p}(\Omega) \to \ell^{p}(\dualLat{L}, w_{0})\) with \(\FTFund f(\kappa) = \FourierCoef{f}(\kappa)\) is bounded, where \(w_{0}(\kappa) = \varphi(\kappa)^{2-p}\). Next, we note that \(\FTFund : L^{p}(\Omega) \to \ell^{p'}(\dualLat{L})\) is also bounded because of the Hausdorff-Young inequality \eqref{eq:HY}.
    Now, choose \(0 \leq \theta \leq 1\) such that \(\frac{1}{b} = \frac{1 - \theta}{p} + \frac{\theta}{p'}\), so \(\theta = \left(\frac{1}{p} - \frac{1}{b}\right)/\left(\frac{1}{p} - \frac{1}{p'}\right)\). Then \cref{th:realInterpolation} tells us that \(\FTFund : L^{p}(\Omega) \to \ell^{b}(\dualLat{L},w)\) is bounded, where
    \[w(\kappa) = w_{0}(\kappa)^{\frac{b(1-\theta)}{p}} = \left(\varphi(\kappa)^{\frac{1}{b} - \frac{1}{p'}}\right)^{b}.\]
    Here we note that \cref{th:realInterpolation} ensures that the hidden constant in the Hausdorff-Young-Paley inequality only depends on \(p\) and \(b\), since the dependence on \(\varphi(\kappa)\) is contained in the factor \(M_{\varphi}^{\frac{1}{b} - \frac{1}{p'}}\).
\end{proof}

\section{\texorpdfstring{\(L^p\)}{Lp}-\texorpdfstring{\(L^q\)}{Lq} boundedness of Fourier multipliers for \texorpdfstring{\(1 < p,q < \infty\)}{1 < p,q < infty}}\label{section:BoundFourierMult}

We are finally ready to accomplish our main goal, namely establishing the \(L^{p}\)-\(L^{q}\) boundedness of Fourier multipliers on fundamental domains of lattices in \(\R^{d}\). In this section we prove this for \(1 < p,q < \infty\) under a Hörmander-type condition, which is inspired by Lars Hörmander's result \cite[Theorem 1.11]{Hormander1960}.

We start by defining Fourier multipliers, which intuitively correspond to multiplication by a function in the (multidimensional) `frequency space'.

\begin{definition}
    An operator \(A : C^{\infty}(\Omega) \to C^{\infty}(\Omega)\) is called a \emph{Fourier multiplier} with symbol \(\sigma : \dualLat{L} \to \C\) if for all \(f \in C^{\infty}(\Omega)\) and \(\kappa \in \dualLat{L}\) it holds that
    \begin{equation}\label{eq:FourierMultProp}
        \FourierCoef{Af}(\kappa) = \sigma(\kappa) \, \FourierCoef{f}(\kappa).
    \end{equation}
\end{definition}

\begin{example}
    A typical example of a Fourier multiplier is any linear partial differential operator with constant coefficients, i.e.\ \(A = \sum_{\abs{\alpha} \leq k} a_{\alpha} \partial^{\alpha}\) for some \(k \in \N\), as the Fourier transform converts derivatives into polynomials. Here we used for \(\alpha \in \N^{d}\) the multi-index notations \(\partial^{\alpha} = \partial_{x_{1}}^{\alpha_{1}} \dots \partial_{x_{d}}^{\alpha_{d}}\) and \(\abs{\alpha} = \alpha_{1} + \dots + \alpha_{d}\).
\end{example}

\begin{remark}
    Since \(\FTFund : C^{\infty}(\Omega) \to \schwartzSpace(\dualLat{L})\) is invertible, a Fourier multiplier \(A\) satisfies the relation
    \begin{equation}\label{eq:FourierMultPseudo}
        Af(x) = \FTFund^{-1}[\sigma(\kappa) \, \FourierCoef{f}(\kappa)](x) = \sum_{\kappa \in \dualLat{L}} \sigma(\kappa) \, \FourierCoef{f}(\kappa) \, e^{2 \pi i \kappa \cdot x}.
    \end{equation}
    This is a special case of a pseudo-differential operator on \(\Omega\), which also allows symbols \(\sigma(x,\kappa)\) that depend on both \(x\) and \(\kappa\). Remark that every Fourier multiplier is a linear operator as a consequence of \eqref{eq:FourierMultPseudo}. Moreover, a straightforward verification using \eqref{eq:FourierMultProp} shows that any Fourier multiplier \(A : C^{\infty} \to C^{\infty}\) is continuous if the symbol has at most polynomial growth. One can easily verify that the condition on the symbol in the following theorems always restrict the symbol to be bounded, so we always consider continuous Fourier multipliers.
\end{remark}

We start by treating the case \(p = q = 2\) separately. This is the only case where \(p = q\) that we will discuss.

\begin{theorem}\label{th:boundednessForPTwo}
    Let \(A : C^{\infty}(\Omega) \to C^{\infty}(\Omega)\) be a Fourier multiplier with bounded symbol \(\sigma \in \ell^{\infty}(\dualLat{L})\). Then we have
    \[\operatorNormLpLq{2}{2}{A} := \sup_{f \neq 0} \frac{\LpnormOver{2}{\Omega}{Af}}{\LpnormOver{2}{\Omega}{f}} \leq \lpnormOver{\infty}{\dualLat{L}}{\sigma}.\]
    In particular, \(A\) can be extended to a bounded linear operator from \(L^{2}(\Omega)\) to \(L^{2}(\Omega)\).
\end{theorem}

\begin{proof}
    The Plancherel formula \eqref{Plancherel} yields
    \begin{align*}
        \LpnormOver{2}{\Omega}{Af} = \lpnormOver{2}{\dualLat{L}}{\FourierCoef{Af}} = \left(\sum_{\kappa \in \dualLat{L}} |\sigma(\kappa) \, \FourierCoef{f}(\kappa)|^{2}\right)^{\frac{1}{2}} &\leq \lpnormOver{\infty}{\dualLat{L}}{\sigma} \lpnormOver{2}{\dualLat{L}}{\FourierCoef{f}} \\ &= \lpnormOver{\infty}{\dualLat{L}}{\sigma} \LpnormOver{2}{\Omega}{f}.
    \end{align*}
    The result now follows easily.
\end{proof}

Next, we examine the case that Hörmander considered in \cite[Theorem 1.11]{Hormander1960}. Our proof is similar to that of \cite[Theorem 4.8]{CKNR}.

\begin{theorem}\label{th:boundedness1}
    Let \(1 < p \leq 2 \leq q < \infty\) with \(p\) and \(q\) not both equal to \(2\), and let \(A : C^{\infty}(\Omega) \to C^{\infty}(\Omega)\) be a Fourier multiplier with symbol \(\sigma\) satisfying
    \begin{equation}\label{symbolGrowthCondition}
        \sup_{s > 0} s \left(\sum_{\substack{\kappa \in \dualLat{L} \\  \abs{\sigma(\kappa)} \geq s}} 1 \right)^{\frac{1}{p} - \frac{1}{q}} < \infty.
    \end{equation}
    Then
    \[\operatorNorm{A} := \sup_{f \neq 0} \frac{\LpnormOver{q}{\Omega}{Af}}{\LpnormOver{p}{\Omega}{f}} \lesssim  \sup_{s > 0} s \left(\sum_{\substack{\kappa \in \dualLat{L} \\  \abs{\sigma(\kappa)} \geq s}} 1 \right)^{\frac{1}{p} - \frac{1}{q}}.\]
    In particular, \(A\) can be extended to a bounded linear operator from \(L^{p}(\Omega)\) to \(L^{q}(\Omega)\).
\end{theorem}

\begin{proof}
    First, we assume that \(p \leq q'\). Note that \(q' \leq 2\). Hence the Hausdorff-Young inequality \eqref{eq:HYInv} implies that
    \[\LpnormOver{q}{\Omega}{Af} = \LpnormOver{q}{\Omega}{\FTFund^{-1}(\FourierCoef{Af})} \leq \lpnormOver{q'}{\dualLat{L}}{\FourierCoef{Af}} = \left(\sum_{\kappa \in \dualLat{L}} |\sigma(\kappa) \, \FourierCoef{f}(\kappa)|^{q'}\right)^{\frac{1}{q'}}.\]
    Set \(\frac{1}{r} = \frac{1}{p} - \frac{1}{q} = \frac{1}{q'} - \frac{1}{p'}\). An application of the Hausdorff-Young-Paley inequality \eqref{ineq:HYP} with \(b := q\) and \(\varphi(\kappa) = \abs{\sigma(\kappa)}^{r}\), where we note that \(1 < p \leq q' \leq p' < \infty\), yields
    \begin{align*}
        \LpnormOver{q}{\Omega}{Af} \leq \left(\sum_{\kappa \in \dualLat{L}} |\sigma(\kappa) \, \FourierCoef{f}(\kappa)|^{q'}\right)^{\frac{1}{q'}} &= \left(\sum_{\kappa \in \dualLat{L}} |\FourierCoef{f}(\kappa) \, \varphi(\kappa)^{\frac{1}{b} - \frac{1}{p'}}|^{q'}\right)^{\frac{1}{q'}} \\
        &\lesssim \left(\sup_{s > 0} s \sum_{\substack{\kappa \in \dualLat{L} \\ \varphi(\kappa) \geq s}} 1\right)^{\frac{1}{r}} \LpnormOver{p}{\Omega}{f}.
    \end{align*}
    Hence we find
    \begin{align*}
        \operatorNorm{A} \lesssim \left(\sup_{s > 0} s \sum_{\substack{\kappa \in \dualLat{L} \\ \abs{\sigma(\kappa)}^{r} \geq s}} 1\right)^{\frac{1}{r}} &= \left(\sup_{s > 0} s^{r} \sum_{\substack{\kappa \in \dualLat{L} \\ \abs{\sigma(\kappa)} \geq s}} 1\right)^{\frac{1}{r}} \\ &= \sup_{s > 0} s \left(\sum_{\substack{\kappa \in \dualLat{L} \\ \abs{\sigma(\kappa)} \geq s}} 1\right)^{\frac{1}{r}},
    \end{align*}
    which is finite by assumption. This completes the proof for the case \(p \leq q'\).
    
    Now we consider the case \(q' \leq p\), which is equivalent to \(p' \leq q = (q')'\). It is well-known \cite[Theorem C.4.50]{PsiDOAndSymm} that the dual space of the Banach space \(L^{p}(\Omega)\) is \(L^{p'}(\Omega)\). General Banach space theory \cite[Theorem 4.10]{FARudin} tells us that the adjoint operator \(\adj{A}\) satisfies \(\adjOperatorNorm{\adj{A}} = \operatorNorm{A}\).
    
    We now prove that \(\adj{A}\) is a Fourier multiplier with symbol \(\complexConjugate{\sigma}\). A general property of an orthonormal basis of a Hilbert space is that it satisfies \emph{Parseval's identity} \cite[Theorem 4.18]{RealComplex}. Considering the orthonormal basis \(\{e^{2 \pi i \kappa \cdot x} \mid \kappa \in \dualLat{L}\}\) of \(L^{2}(\Omega)\) we get for all \(f, g \in L^{2}(\Omega)\) that
    \[\innerproductOmega{f}{g} = \sum_{\kappa \in \dualLat{L}} \FourierCoef{f}(\kappa) \, \FourierCoef{g}(\kappa).\]
    Hence, on the one hand we find for all \(f,g \in C^{\infty}(\Omega) \subseteq L^{2}(\Omega)\) that
    \[\innerproductOmega{Af}{g} = \sum_{\kappa \in \dualLat{L}} \FourierCoef{Af}(\kappa) \, \complexConjugate{\FourierCoef{g}(\kappa)} = \sum_{\kappa \in \dualLat{L}} \sigma(\kappa) \, \FourierCoef{f}(\kappa) \, \complexConjugate{\FourierCoef{g}(\kappa)} = \sum_{\kappa \in \dualLat{L}} \FourierCoef{f}(\kappa) \, \complexConjugate{\complexConjugate{\sigma(\kappa)} \, \FourierCoef{g}(\kappa)},\]
    while on the other hand we have
    \[\innerproductOmega{Af}{g} = \innerproduct{f}{\adj{A}g} = \sum_{\kappa \in \dualLat{L}} \FourierCoef{f}(\kappa) \, \complexConjugate{\FourierCoef{\adj{A}g}(\kappa)}.\]
    We can choose \(\FourierCoef{f}(\kappa) = \delta_{\kappa, \kappa_{0}}\) for any \(\kappa_{0} \in \dualLat{L}\), where \(\delta_{\kappa, \kappa_{0}}\) is a Kronecker delta, since such an \(f\) belongs to \(\schwartzSpace(\dualLat{L})\) and thus can be inverted by the Fourier transform. Hence, we find for all \(\kappa \in \dualLat{L}\) that
    \[\FourierCoef{\adj{A}g}(\kappa) = \complexConjugate{\sigma(\kappa)} \, \FourierCoef{g}(\kappa),\]
    which shows indeed that \(\adj{A}\) is a Fourier multiplier with symbol \(\complexConjugate{\sigma}\).
    
    Of course, \(\abs{\sigma(\kappa)} = |\complexConjugate{\sigma(\kappa)}|\) for all \(\kappa \in \dualLat{L}\) so that \(\complexConjugate{\sigma}\) also fulfills the growth condition \eqref{symbolGrowthCondition}. Thus, if we apply the first case to the Fourier multiplier \(\adj{A}\) with \(1 < q' \leq 2 < p' < \infty\), we obtain
    \begin{align*}
        \operatorNorm{A} = \adjOperatorNorm{\adj{A}} & \lesssim \sup_{s > 0} s \left(\sum_{\substack{\kappa \in \dualLat{L} \\ |\complexConjugate{\sigma(\kappa)}| \geq s}} 1\right)^{\frac{1}{q'}-\frac{1}{p'}} \\
        &= \sup_{s > 0} s \left(\sum_{\substack{\kappa \in \dualLat{L} \\ |\sigma(\kappa)| \geq s}} 1\right)^{\frac{1}{p}-\frac{1}{q}},
    \end{align*}
    where we note again that \(\frac{1}{p} - \frac{1}{q} = \frac{1}{q'} - \frac{1}{p'}\).
\end{proof}

The finiteness of the measure of \(\Omega\) allows us to prove the \(L^{p}\)-\(L^{q}\) boundedness of Fourier multipliers in a broader range for \(p\) and \(q\). This idea and our proof originate from \cite[Corollary 4.9]{CKNR}.

\begin{theorem}\label{th:boundedness2}
    Let \(1 < p,q < \infty\), and let \(A : C^{\infty}(\Omega) \to C^{\infty}(\Omega)\) be a Fourier multiplier with symbol \(\sigma\). Assuming that the upper bound in each case is finite, we get the following inequalities.
    \begin{enumerate}
        \item If \(1 < p,q \leq 2\), then
        \[\operatorNorm{A} \lesssim \sup_{s > 0} s \left(\sum_{\substack{\kappa \in \dualLat{L} \\  \abs{\sigma(\kappa)} \geq s}} 1 \right)^{\frac{1}{p} - \frac{1}{2}}.\]
        \item If \(2 \leq p,q < \infty\), then
        \[\operatorNorm{A} \lesssim \sup_{s > 0} s \left(\sum_{\substack{\kappa \in \dualLat{L} \\  \abs{\sigma(\kappa)} \geq s}} 1 \right)^{\frac{1}{q'} - \frac{1}{2}}.\]
    \end{enumerate}
    In particular, \(A\) can be extended to a bounded linear operator from \(L^{p}(\Omega)\) to \(L^{q}(\Omega)\).
\end{theorem}

\begin{proof}
    Consider first the case \(1 < p,q \leq 2\). We assume that \(q \neq 2\), because the other case is already treated in the previous theorems. Choose \(\widetilde{q} \geq 1\) such that \(\frac{1}{2} + \frac{1}{\widetilde{q}} = \frac{1}{q}\), i.e.\ \(\widetilde{q} = \frac{2q}{2-q}\). For all \(f \in C^{\infty}(\Omega)\) we have
    \[\LpnormOver{q}{\Omega}{f} \leq \LpnormOver{\widetilde{q}}{\Omega}{1} \LpnormOver{2}{\Omega}{f} = \meas{\Omega}^{\frac{1}{\widetilde{q}}} \LpnormOver{2}{\Omega}{f} = \meas{\Omega}^{\frac{2-q}{2q}} \LpnormOver{2}{\Omega}{f}\]
    by (a simple generalisation of) Hölder's inequality. Hence we find
    \[\operatorNorm{A} = \sup_{f \neq 0} \frac{\LpnormOver{q}{\Omega}{Af}}{\LpnormOver{p}{\Omega}{f}} \lesssim \sup_{f \neq 0} \frac{\LpnormOver{2}{\Omega}{Af}}{\LpnormOver{p}{\Omega}{f}} = \operatorNormLpLq{p}{2}{A}.\]
    \Cref{th:boundedness1} yields
    \[\operatorNorm{A} \lesssim \operatorNormLpLq{p}{2}{A} \lesssim \sup_{s > 0} s \left(\sum_{\substack{\kappa \in \dualLat{L} \\  \abs{\sigma(\kappa)} \geq s}} 1 \right)^{\frac{1}{p} - \frac{1}{2}}.\]
    
    For the other case, namely \(2 \leq p,q < \infty\), we have \(1 < p',q' \leq 2\) so that the first case gives that
    \[\operatorNorm{A} = \adjOperatorNorm{\adj{A}} \lesssim \sup_{s > 0} s \left(\sum_{\substack{\kappa \in \dualLat{L} \\  \abs{\sigma(\kappa)} \geq s}} 1 \right)^{\frac{1}{q'} - \frac{1}{2}},\]
    where we have used the same properties of the adjoint operator \(\adj{A}\) as in the proof of \cref{th:boundedness1}.
\end{proof}

\begin{remark}
    In fact, we can cover all cases where \(1 < p,q < \infty\). The only remaining case is \(1 < q \leq 2 \leq p < \infty\). Suppose \(A : C^{\infty}(\Omega) \to C^{\infty}(\Omega)\) is a Fourier multiplier with symbol \(\sigma\) satisfying the growth condition
    \[\sup_{s > 0} s \left(\sum_{\substack{\kappa \in \dualLat{L} \\  \abs{\sigma(\kappa)} \geq s}} 1 \right)^{\frac{1}{q} - \frac{1}{p}} < \infty.\]
    Note that in this case \(1 < p' \leq 2 \leq q' < \infty\). Since \(\frac{1}{p'} - \frac{1}{q'} = \frac{1}{p} - \frac{1}{q}\), \(\complexConjugate{\sigma}\) satisfies the growth condition \eqref{symbolGrowthCondition} so that \cref{th:boundedness1} yields \(\adj{A} \in \contLinLpLq{p'}{q'}\). Thus we obtain \(\adj{(\adj{A})} = A \in \contLinLpLq{q}{p}\).
\end{remark}

\begin{example}
    We return to the example of linear partial differential operators with constant coefficients. Note that derivatives of order at least 1 are mapped to polynomials of degree at least 1 under the Fourier transform. Since the growth condition for the symbol in the former theorems restrict the symbol to be bounded, the only such operator that can fulfil the growth condition for the symbol in the former theorems is an operator of the form \(Af(x) = a_{0} f(x)\) for some \(a_{0} \in \C\). In fact, this operator only fulfils the requirements for \(L^{2}(\Omega)\)-\(L^{2}(\Omega)\) boundedness.
    
    This is not surprising. Otherwise, there would be a canonical way to extend the partial derivatives to \(L^{p}\)-spaces, which we do not expect.
\end{example}

\begin{example}
    We now give another example where our developed theory does apply. Consider the operator \(A : C^{\infty}(\Omega) \to C^{\infty}(\Omega)\) defined by
    \[Af(x) = \sum_{\kappa \in \dualLat{L}} e^{-\abs{\kappa}^{2}} \, \FourierCoef{f}(\kappa) \, e^{2 \pi i \kappa \cdot x}.\]
    Clearly, this is a Fourier multiplier with symbol \(\sigma(\kappa) = e^{-\abs{\kappa}^{2}}\). Note that the series converges absolutely and uniformly since both \(e^{-\abs{\kappa}^{2}}\) and \(\FourierCoef{f}(\kappa)\) belong to \(\schwartzSpace(\dualLat{L})\). We can easily verify that this operator has a bounded \(L^{p}(\Omega)\)-\(L^{q}(\Omega)\) extension for all \(1 < p,q < \infty\) if we use the inequality
    \[\#\{\kappa \in \dualLat{L_{A}} : \abs{\kappa} \leq R\} \leq \frac{2R+1}{\abs{a_{1} \cdot e_{1}}} \dots \frac{2R+1}{\abs{a_{d} \cdot e_{d}}} = C_{A} (2R+1)^{d},\]
    where \(\#\) denotes the cardinality (function), \(R > 0\), \(e_{1}, \dots, e_{d}\) is the standard basis of \(\R^{d}\), \(C_{A} = \prod_{j=1}^{d} \frac{1}{\abs{a_{j} \cdot e_{j}}}\) and \(A = (a_{1} \dots a_{d})\) is the generator matrix of the lattice \(L_{A}\). We assume that none of the inner products \(a_{j} \cdot e_{j}\) is zero, and this is always possible. This inequality generalises the idea that we can bound the number of lattice points of \(\Z^{d}\) in the ball \(B(0,R)\) by the number of lattice points in \([-R,R]^{d}\), but instead of enlarging \(B(0,R)\) to a cube \([-R,R]^{d}\), we enlarge it to a parallelotope.
    
    Of course, one can find a lot of other examples, but the former inequality may come in handy when one verifies the growth condition on the symbol.
\end{example}

\section*{Acknowledgements}

I would like to thank Prof.\ Michael Ruzhansky for suggesting this research topic and for his guidance. I am also grateful to Dr.\ Vishvesh Kumar for helpful discussions.
 
\printbibliography

\end{document}